\documentclass[12pt]{article} 
\usepackage{amssymb,epsfig,amsmath,amsthm,cite}

\usepackage[dvipsnames]{xcolor}

\newcommand{\trieq}{\stackrel{\triangle}{=}}

\newcommand{\cF}{{\mathcal F}}

\newcommand{\E}{{\mathbb E}}
\newcommand{\PP}{{\mathbb P}}
\newcommand{\R}{{\hbox{I{\kern -0.22em}R}}}
\newcommand{\sR}{\rm I{\kern -0.22em}R}
\newcommand{\id}{{\mathbf 1}}

\newcommand{\bB}{{\mathbb B}}
\newcommand{\bM}{{\mathbb D}}
\newcommand{\bS}{{\mathbb S}}

\newtheorem{MyDef}{Definition}[section]
\newtheorem{theorem}{Theorem}[section]
\newtheorem{thrm}{Theorem}[section]
\newtheorem{prop}{Proposition}[section]
\newtheorem{exam}{\sc Example} [section]
\newtheorem{Ex}{\sc Example} [section]
\newtheorem{assump}{Assumption}[section]
\newtheorem{re}{Remark}[section]
\newtheorem{coro}{Corollary}[section]
\newtheorem{lemma}{Lemma}[section]

\providecommand{\keywords}[1]
{
  \small	
  \textbf{\textit{Keywords---}} #1
}

\title{General Optimal Stopping without Time Consistency}
\author{Hanqing Jin and Yanzhao Yang\thanks{Mathematical Institute and Oxford-Octa Laboratory in Digital Economics, the University of Oxford, Woodstock Road, Andrew Wiles Building, Oxford OX2 6GG, UK. Email:
$<$jinh,yanzhao.yang@maths.ox.ac.uk$>$.}}

\begin{document}


\maketitle

\begin{abstract}
In this paper, we propose a new framework for solving  a general dynamic optimal stopping problem 
without time consistency. A  sophisticated solution is proposed and  is well-defined for any time setting with general flows of objectives. A backward iteration is proposed to find the solution. The iteration works  with an additional condition, 
which holds in interesting cases including the time inconsistency arising from  non-exponential discounting.  Even if the iteration does not work, the equilibrium solution can still be studied by a forward definition. 
\end{abstract}


Mathematics Subject Classification 60G40.91B06  

JEL classification D81.D90.C62

\section{Introduction}

As a branch of applied probability, optimal stopping problems have been
widely studied  and applied in different areas such as statistics, 
economics, and finance.   In general,  an  optimal stopping problem is to 
maximise some objective in a 
{\it dynamic system} by choosing a proper {\it stopping time}.   A dynamic 
system is modelled by its state process  $X_t$ together with an information 
flow described by a filtration $\cF_t$ (to which $X_\cdot$ is adapted). Starting from time 
$t=0$,  
the optimal stopping problem is to find an 
$\cF_t$-stopping time $\tau$ to maximise the interest of an agent with objective function $J(\tau)$.

 Most related research in literature is on continuous time optimal stopping, where 
 the dynamic system is modelled by a Markovian diffusion process 
 $$dX_t=\mu(t, X_t)dt+\sigma(t, X_t)dB_t$$
 with a standard Brownian motion $B$ and some deterministic functions $\mu$ and $\sigma$, 
 and the information flow is the observation of the state process $X$ or that of the Brownian motion $B$, 
 i.e. $\cF_t=\sigma(X_s: s\le t)$ or $\cF_t=\sigma(B_s: s\le t)$.  The  objective to maximise is often set in the form 
 $J(\tau)=\E[ f( X_\tau)]$ with some real function $f(\cdot)$, and hence the optimal stopping problem is formulated as
\begin{equation}\label{classtop}
 \max_{\tau\in \Gamma(0,T)}\E[ f(X_\tau)],
 \end{equation}
where $\Gamma(0,T)\trieq\{\mbox{ stopping time  } \tau : 0\le \tau\le T\}$\footnote{ When $T=+\infty$,
$\Gamma(0,+\infty)\trieq\{\mbox{ stopping time  }\tau:  \tau\in [0,+\infty)\}$.}
\footnote{The form $\E [f(X_\tau, \tau)]$ is more 
 popular, but this is a special case of $\E [f(X_\tau)]$ taking the time 
$t$ as an additional dimension of  the state process.}. Exceptions include  Toit  and Peskir \cite{peskir-toit},  
 Shiryaev, Xu and Zhou \cite{shiryaev2008thou}
 \footnote{The problems in these two papers 
 can be  transformed into standard ones as above.} and Xu and Zhou \cite{xu2013optimal}.   
 
 To study the optimal stopping problem as above, dynamic programming is the most often used.
In this theory, the optimisation at time $t=0$ is expanded into a flow of optimisation problems 
starting from any time $t<T$ with the objective $J(\tau; t,x )=\E[f(X_\tau)|X_t=x]$. With this expansion,
the dynamic principle holds, which reads ``an optimal stopping time picked up at time $t=0$
is also optimal at any future time $t$ before the arrival of the optimal stopping time''.   We interpret this principle as the time consistency of the preference. 

In a general dynamic stopping problem, due to the evolution of the environment and the player's taste, 
the objective can change in some inconsistent style, i.e. the dynamic principle may fail.  
Inconsistent change in objectives  is common in dynamic decision making, as shown in the following examples from daily life. 
\begin{exam}\label{smoking}
 Smokers often try to quit smoking for some health concern. The choice of  quitting time 
 can be regarded as an optimal stopping problem. In reality, a common optimal plan is in 
 the form like  ``to quit tomorrow'', which shows the smoker will keep changing the optimal 
 stopping time.  
\end{exam}

\begin{exam}\label{stocksale}
When investors holding some shares plan to sell them in the future, they often set a
return target, like 20\%, and sell the shares when the target is achieved. Sometime they
also set a stop-loss level in terms of return rate. But many investor will keep the return target
when time goes on, i.e. they will ask for more profit when the share price goes up.  
\end{exam}

If we don't ignore the time inconsistency, then we should take our future preferences (and hence the resulted decisions)
into account in making the current decision. Without  time consistency,  the optimal solution obtained at time $t=0$ may not be optimal in the future, hence the concept of classical optimal solution does not make sense in general, and we need to re-define 
the solution to replace the concept of optimal solution. 

For optimal stopping without time consistency, the 
so-called naive policy makes decision as follows.
At any time $t$, the decision of optimal stopping is the comparison of immediate stopping and waiting for future decisions. 
The player should stop immediately if $J(t;t)>J(\tau;t)$ for any stopping time $\tau>t$, and 
wait otherwise. 
So the naive solution can be defined as the first time that immediate stopping is optimal in this sense. 

 
Economists has studied 
time inconsistency in dynamic optimal control. Strotz \cite{strotz1955myopia} studied the optimal cumulated utility with non-exponential discounting, and proposed to deal with time inconsistency by replacing the non-exponential discounting with an exponential one. Pollak \cite{pollak1968consistent} proposed a backward induction to 
get a consistent solution for a discretised version of  the problem in \cite{strotz1955myopia}, and suggested making the discretisation finer and finer to converge to a continuous time solution. Peleg and Yaari \cite{peleg1973existence} worked further on the same problem, and propose the equilibrium solution for the problem. 
More recently, Grenadier and Wang \cite{Grenadier2006InvestmentUU} studied the equilibrium behaviour for investment under time inconsistency. 
For optimal stopping problem, O'Donoghue and Rabin \cite{Donoghue-rabin} proposed a solution for a general stopping problem without time consistency in a finite discrete-time setting, which respects the idea of the equilibrium solution in \cite{peleg1973existence}.
  
In a finite discrete-time setting with time spots $t\in\{0,1,\cdots, n\}$,  we reformulate the definition of the solution to our stopping problem  in \cite{Donoghue-rabin} as follows. 
\begin{itemize}
\item [(1)] At time $t=n$, $\tau_n^b=n$.
\item [(2)] At any time $t<n$, 
$$\tau_t^b=\left\{\begin{array}{cc} t & \mbox{ if } J(t;t)> J(\tau_{t+1}^b; t)\\ \tau^b_{t+1}& \mbox{otherwise}\end{array}\right. .\footnote{When $J(t;t)=J(\tau^b_{t+1};t)$ happens, the solution $\tau^b_t$ can be defined as either $t$ or $\tau^b_{t+1}$, 
the former refers  to the minimal solution, and the latter refers to the maximal solution. In this paper, we use the maximal version.}$$ 
\item [(3)] $\tau^b_t$ is defined as the sophisticated stopping time at time $t$, which we call the {\it backward induction} solution (BIS in short) in this paper. 
 \end{itemize}
 The sophisticated stopping time takes into account the future preferences at every time $t$, 
 which sounds more rational.
 It works perfectly in a finite discrete-time system.
 When the set of time spots is infinitely discrete or continuous, or even a hybrid of discrete and continuous, there is no good definition of a proper solution accepted widely in the literature. 

There are different ideas to re-define the solution for the time inconsistent optimal stopping problem in a continuous time setting with specific form of objectives. 
Ebert et al \cite{ebert-wei} borrowed the concept of equilibrium control from the dynamic control problem without time consistency, and formulated the optimal stopping problem as a special optimal control problem with the control space consisting of $0$ and $1$. 
 Christensen and Lindensj\"{o} \cite{Christensen} studied the stopping problem with state dependent preferences. They defined the equilibrium by a direct perturbation on the candidate of stopping time, and tried to find the equilibrium by a backward iteration.  

With the BIS 
for finite discrete time setting, researchers also proposed to tackle the continuous-time problem by discretising the time interval, getting the backward induction solutions, and taking their limit as the equilibrium solution.  It was first employed by Pollak \cite{pollak1968consistent} to study the consumption problem of Strotz \cite{strotz1955myopia} and has recently been revisited and extensively studied by a series of papers, for instance, Yong \cite{Yong2011DeterministicTO} and Wang and Yong \cite{Wang2019TimeinconsistentSO}. 
For this approach, it is still unknown that whether it is possible to establish the convergence of limiting equilibrium in a partition-free manner.

Recently,  Huang and Nguyen-Huu \cite{huang2018time} and Huang  et al \cite{huang2020general} 
 introduced a new definition of equilibrium solution by iteration. 
Some  iterative procedures were introduced to capture how the future decision making should be taken into account at the current time. 
They solved some problems without  time consistency, mainly  in the class of  Markovian system and non-exponential discounting reward. They also left some open questions on the general existence of a solution.

With time consistency, an optimal solution to an optimal stopping problem is clearly defined without any structure of the time setting or a specific form of the preference. We believe that  there should also be a structure-free solution to the problem without time consistency. 

In this paper, we propose a structure-free solution for a general stopping problem, and propose an iterative procedure to search for the solution. Our solution is different from the one 
proposed by Huang and Nguyen-Huu \cite{huang2018time}  and has some connection with the one proposed by 
\cite{Christensen}.  

This paper will be presented as follows. 

In Section \ref{sec2}, we review the classical equilibrium strategy BIS defined above, 
and prove some properties of this classical equilibrium.

In Section \ref{sec3}, we propose our general framework which is referred as Sophisticated Stopping Strategy and we propose a meaningful backward iteration to approach our equilibrium, and suggest a sufficient condition making the iteration work. 

In Section \ref{sec6}, we provide a counter example for our backward iteration proposed in Section \ref{sec3}, and show that it can be endless without more structure of the problem. We also prove that the backward iteration works for a general optimal stopping problem with non-exponential discounting. 

In Section \ref{sec5}, we show by counter an example that our equilibrium does not respect the classical one proposed by Strotz, and we find that our equilibrium satisfies some monotonicity with respect to the time horizon of the problem, while the classical equilibrium does not. 

Finally we conclude the paper in Section \ref{sec8}.

\section{Problem Formulation and Finite Discrete Time Setting}\label{sec2}
\newcommand{\TS}{{\mathbb T}}
\newcommand{\ST}{{\mathcal T}}

In the whole paper, we study a dynamic optimal stopping problem in a probability space $(\Omega,\mathcal{F},\mathbb{P})$ over a set of time points $\TS$ equipped with filtration 
$\{\cF_t: t\in \TS\}$.  Denote  by $t=0$ the starting time and assume $0\in \TS$. Denote $T$ as the terminal time which can be finite or infinite. 
For any $\rho\in \ST(0)$, we can define 
$$\ST(\rho):=\{\tau:  \tau \mbox{ is a stopping time, } \tau\ge \rho\}.$$

At any time $t\in \TS$, we aim to maximize the objective $J(\tau; t)$ over $\tau\in \ST(t)$,  where $J(\cdot; t)$ is a mapping from $\ST(t)$ to the set of $\cF_t$-measurable random variables $L^0(\Omega, \cF_t, \R)$, which satisfies\footnote{In literature  on general optimal stopping problem,  another mild condition has been proposed on the preference $J$, which reads: 
$\forall\, \tau_{1},\tau_2 \in \ST(t)$, $A \in \mathcal{F}_t$, $J(\tau_1\mathbb{I}_{A}+\tau_2\mathbb{I}_{A^{c}};t)=J(\tau_1;t)\mathbb{I}_{A}+J(\tau_2;t)\mathbb{I}_{A^c}$. In our general theorem in Section \ref{sec3}, we don't need this condition. 
}: 
\begin{itemize}
    \item [(i)] $J(\tau; t)$ is $\mathcal{F}_t$ measurable for any $\tau\in \ST(t)$ ;
    \item [(ii)] $J(\tau_1;t)=J(\tau_2;t)$ if $\tau_1=\tau_2$ a.s.;
\end{itemize}

In most of the literature on optimal stopping, there is a state process $X$ to indicate the system, where the optimization is studied.  
In this paper, we aim to define a sophisticated solution to the time inconsistent problem, and we don't plan to deal with explicit form of the solution, so we don't need to emphasize the role of the state process, and all information from the state process will be contained in the filtration $\cF$. 

Let us take the finite discrete time setting as an example, where  $\TS=\{0,1,2,\cdots, T\}$ for some integer $T$. For general filtration $\{\cF_t: t\in \TS\}$ and preference flow $\{J(\cdot; t): t\in\TS\}$, we have defined the ``sophisticated optimal'' solution  $\tau^b$ in the last section  by a sequence of stopping times $\{\tau^b_t: t\in \TS\}$. In our section of Introduction, we call it the ``backward induction solution'' (BIS).  

The BIS is well-defined in the finite discrete time setting, but the definition is not very convenient to extend to a general setting. 
We will re-define it by iteration, which can shed a light on the generalization. To this end, 
we define the sequence of stopping times $\{\tau_k^{be}: k=0,\cdots, T\}$ 
as follows. 
\begin{MyDef}
\begin{itemize}
\item [(i)] $\tau^{be}_T=T$;
\item [(ii)] For any $t<T$, $\tau^{be}_t=\inf\{k\ge t: J(k;k)> J(\tau^{be}_{k+1};k)\}\wedge \tau^{be}_{t+1}$;
\end{itemize}
\end{MyDef}
Obviously, 
$\tau^{be}_t$ increases in $t$. 
We claim that 
\begin{theorem}
$\tau^{be}_k=\tau^b_k$.
\end{theorem}
\begin{proof}
We prove this by induction. 
The statement is trivially true for $k=T$ and $k=T-1$.
Suppose that the statement holds for all $k\ge  n$ for some $n<T$, it suffices to show that the statement  also holds for $k=n-1$.

At $k=n-1$, by the new definition we know 
\begin{equation}
  \tau^{be}_{n-1}=\inf\{i\geq n-1:J(i;i)> J(\tau^{b}_{i+1};i )\}\wedge \tau^{be}_n  
\end{equation}
If
    $J(n-1;n-1)> J(\tau^b_{n};n-1),$
then it is true that
$\tau^{be}_{n-1}=n-1=\tau^b_{n-1};$
 otherwise, we have 
     $J(n-1;n-1)\le J(\tau^b_{n};n-1),$
 which implies  $\tau^b_{n-1}=\tau^b_n$, hence 
 \begin{eqnarray*}
   \tau^{be}_{n-1}&=&\inf\{i\geq n-1: J(i;i)> J(\tau^{be}_{i+1};i)\} \wedge \tau^{be}_n\\
   &=&\inf\{i\geq n: J(i;i)> J(\tau^{be}_{i+1};i)\} \wedge \tau^{be}_n\\
   &=&\tau^{be}_n=\tau^b_n,
 \end{eqnarray*}
  and then $\tau^{be}_{n-1}=\tau^b_{n-1}$ also holds. 
\end{proof}

The equivalent definition $\tau^{be}_t$ reveals that BIS $\tau^b_t$ is the first time after $t$ that immediate stopping is better than stopping at $\tau^b_{t+1}$, where $\tau^{b}_{t+1}$ is  the consequent stopping time of not stopping at time $t$. 
 
The equivalent definition also reveals the following property of BIS. 
 \begin{theorem}\label{bisprop2}
For any $t<T$, any $s$ with $t\le s< \tau^b_t$, we have 
$$J(\tau^b_t; s)\ge J(s;s). $$
\end{theorem}
\begin{proof}
By the definition of $\tau^b_t$, we know that $\tau^b_\nu\le \tau^b_t$ for any $\nu\le t$. 
Furthermore, at time $\tau^b_t$, we should stop immediately, so $\tau^b_{\tau^b_t}=\tau^b_t$. 

For any $s$ with $t\le s<\tau^b_t$, we have 
$s<\tau^b_t\le \tau^b_s\le \tau^b_{\tau^b_t}=\tau^b_t$, hence $J(s;s)\le J(\tau^b_{s+1}; s)$ and 
$\tau^b_t=\tau^b_s=\tau^b_{s+1}$, 
which implies $J(s;s)\le J(\tau^b_s;s).$
\end{proof}

With Theorem \ref{bisprop2}, we know that, starting from time $t$, we should not stop before the solution $\tau^b_t$, which 
is planned to be the final stopping time in the future. We call a stopping time $\tau$ ``approachable''\footnote{Here we only give a vague description of approachability. There are different accurate definitions of approachability, including the equilibrium defined in Huang and Nguyen-Huu \cite{huang2018time} and the one we will define in Section \ref{sec3}.}  by the property that we should not stop at any time before $\tau$. In this sense, $\tau^b_t$ is approachable from time $t$. 
Generally,  $\tau^b_t$ is not the only approachable stopping time from $t$. In fact, from time $t$, the smallest approachable stopping time from $t$ is  $t$ itself. 
Is it possible that $\tau^b_t$ the biggest approachable stopping time? In Section \ref{sec3}, we will give a detailed definition of approachability, and show by an example in Section \ref{sec5} that $\tau^b_t$ is not always the biggest approachable stopping time. 

BIS works only in the finite discrete-time setting. In the rest of this paper, we will study the solution to the optimal stopping time in a general time setting by borrowing ideas and properties of BIS.  Notice that the equivalent definition $\tau^{be}_t$ sounds easier to extend. However, in the definition of $\tau^{be}_t$, at any time $k$ before the next stopping time $\tau^{be}_{t+1}$, we only compare immediate stopping against 
stopping at $\tau^{be}_{k+1}$, while stopping times between $k$ and $\tau^{be}_{k+1}$ are ignored, which sounds not 
natural in terms of optimisation over all stopping times. 
We are going to revise it, and include all stopping times in between into the comparison, which will make the final solution different, and more natural to us.

\section{Sophisticated Stopping in General Time Setting}\label{sec3}
Denote $T=\sup\{t: t\in \TS\}$, with which $T=+\infty$ when $\TS$ is unbounded. 

For a general optimal stopping problem, the decision at any time $t$ is 
either immediate stopping or waiting for decisions in the future. Put decisions at different time together,   the solution  stopping time should be the first time that immediate stopping is the best. 
This interpretation of a solution applies to any time setting and regardless the time consistency of the preferences. It contains two requirements on the optimal solution $\tau_*$: 
\begin{itemize}
\item [(i)] We should not stop at any time $t$ before $\tau_*$. 
\item [(ii)] For any stopping time $\tilde \tau$ strictly later than $\tau_*$, there exists some stopping time $\bar \tau$
with $\tau_*\le \bar\tau<\tilde\tau$, such that standing at time $\bar \tau$,  immediate stopping is optimal over all other stopping times before $\tilde \tau$. 
\end{itemize}

We need to describe these two requirements more accurately, and hopefully they together define a unique solution, which we name the {\it sophisticated solution} or {\it equilibrium solution}.

We interpret the first requirement above by the concept of approachability, and the second requirement by delimiting. 
We introduce the following mapping from $\ST(0)$ to itself.

\begin{MyDef}\label{stopping time mapping}
For $\rho \in \ST(0)$, define the mapping $F$ by 
 $$F(\rho)=\inf\{t<\rho: J(t;t)> J(\tau;t),\forall \rho \geq \tau>t\}\wedge \rho.$$
\end{MyDef}

Roughly speaking, $F(\rho)$ is the first time $t$ before $\rho$ that immediate stopping is strictly better than  all other stopping time before $\rho$. Mathematically, to ensure that $F(\rho)$ is a stopping time, we need the filtration to be right continuous.

\begin{lemma}\label{cont-J-0}
If the filtration $\{\mathcal{F}_{t}\}_{t\in \TS}$ is right continuous, then $F(\rho)$ is a stopping time for  $\forall \rho\in \ST(0)$.  
\end{lemma}
\begin{proof}
For any constant time $\bar t\ge 0$
the event $\{\omega: \inf\{t<\rho: J(t;t)> J(\tau;t),\forall \rho \geq \tau>t\}\le \bar t\}$ is equivalent to  
$$
\bigcap_{ \varepsilon>0, \varepsilon\in {\mathbb Q}}\{ \exists \,s\le \bar t+\varepsilon, \;  J(s;s)>J(\tau;s) \mbox{ for }\forall\, \rho\geq \tau >s \} 
\in\cF_{\bar t+}=\cF_{\bar t}.
$$
Furthermore, $\{\rho\le \bar t\}\in \cF_{\bar t}$. 
Hence 
$$\{F(\rho)\le \bar t\}=\{\omega: \inf\{t<\rho: J(t;t)> J(\tau;t),\forall \rho \geq \tau>t\}\le \bar t\}\cup\{\rho\le \bar t\}\in \cF_{\bar t}.$$
\end{proof}

We assume that $\{\cF_t\}_{t\in \TS}$ is right continuous in the whole paper. 

It is clear that the mapping $F$ is monotone in the following sense. 

\begin{prop}[Monotonicity of $F$]
For any stopping time $\tau\leq \rho$, $F(\tau)\leq F(\rho)$. 
\end{prop}

\subsection{Approachable Time}
\begin{MyDef}[Approachable Time] A stopping time $\tau$ is called  approachable if  $F(\tau)=\tau$.
\end{MyDef}
This definition of approachability is an equivalent description of our first requirement for a solution to the problem: if $\tau$ is a solution to the stopping problem, 
then at any time $t$ before $\tau$, according to the preference at time $t$, immediate stopping is worse than stopping at some future stopping time $\rho$  before $\tau$, which implies 
$F(\tau)=\tau$.

\begin{re}
Notice that when there are multiple choices of optimal solutions, we will choose to stop at  the latest  one if possible,
hence we should not stop immediately if  immediate stopping results the same objective value as stopping at some later stopping time. We call this stopping principle as ``stop later''. 

If we choose to stop immediately when it is one of the best choices (we call it as the principle of ``stop earlier''), then an approachable time $\tau$ should be defined by 
`` $\forall\, t< \tau, \exists\, \tilde \tau\in \ST(t)$ 
with $t<\tilde \tau\le \tau$ such that 
$J(\tilde \tau;t)> J(t;t)$'',
and we need to follow the same principle in the rest of this paper. All conclusions in this section still hold with a little proper change in Definition 
\ref{delimit}, and conclusions in Section \ref{sec3} also hold other than Theorem \ref{infinty-work}.
\end{re}

The set of approachable stopping times is not empty, since $t=0$ is trivially approachable. Ideally, we prefer larger approachable times, hence we have more choices before their arrivals. Hence the largest approachable time
sounds like a suitable candidate of our solution. But is there a largest approachable time? 
  
\begin{prop} If $\tau_{1}$ and $\tau_{2}$ are both approachable times, so is $\tau_1\vee \tau_2$.
\end{prop}
\begin{proof}
Firstly,  $\tau_{1}\vee \tau_{2}$ is a stopping time. 

By monotonicity of $F$, we have: $F(\tau_1)\leq F(\tau_1\vee \tau_2)$ and $F(\tau_2) \leq  F(\tau_1\vee \tau_2)$. Since $\tau_1$ and $\tau_2$ are both approachable, so we have: $\tau_1 \vee \tau_2=F(\tau_1)\vee F(\tau_2)\leq F(\tau_1 \vee \tau_2) \leq \tau_1 \vee \tau_2$. 

Hence, we have $F(\tau_1 \vee \tau_2)=\tau_2 \vee \tau_2$ which verifies that $\tau_1\vee \tau_2$ is approachable.
\end{proof}

\begin{prop}
For any increasing sequence of approachable times $\{\tau_{n}\}_{n=1,2,\cdots}$,
its limit $\tau:=\lim_{n\rightarrow +\infty} \tau_n$ is  approachable.    
\end{prop}
\begin{proof}
Obviously, $\tau$ is a stopping time and we know that $\forall n, \tau_n=F(\tau_n) \leq F(\tau)$ since every $\tau_n$ is approachable and $F$ is monotone.

Therefore, we have: $\tau=\lim_{n\rightarrow +\infty} \tau_n=\lim_{n\rightarrow +\infty} F(\tau_n) \leq F(\tau) \leq \tau$. 

Hence, we have shown that $\tau=F(\tau)$.
\end{proof}

With these properties of approachable times, we are going to prove that the essential supreme of all approachable times is the largest approachable time. To this end, we use the following property of the essential supreme, 
whose proof can be found in \cite{peskir2006optimal}. 

\begin{thrm}\label{esssup}
 Given $\{Z_{\alpha}:\alpha \in \mathbb{I}\}$ be a family of random variables, where the index set $\mathbb{I}$ 
 can be uncountable. There exists a countable subset   $\mathbb{J}$ of $\mathbb{I}$ 
 such that the random variable $Z^{*}:=sup_{\alpha \in \mathbb{J}}Z_{\alpha}$  satisfies the 
 following two properties:
\begin{itemize}
	\item [(i)] $\PP(Z_{\alpha}\leq Z^{*})=1, \forall \alpha \in \mathbb{I}$;
	\item [(ii)] If there exists another random variable $\hat Z$ satisfies (i), then
	$\PP(Z^*\le \hat Z)=1$. 
\end{itemize}
\end{thrm}
In this Theorem, $Z^{*}$ is called the essential supremum of $\{Z_{\alpha}:\alpha \in \mathbb{I}\}$ and is denoted by $Z^{*}={\rm esssup}_{\alpha \in \mathbb{I}}Z_{\alpha}$.

Define 
$$\tau_*:=\rm{esssup} \{\mbox{all approachable times}\}. $$    

\begin{thrm}
$\tau_*$ is  the maximal approachable time. 
\end{thrm}
\begin{proof}
We only need to prove that $\tau_*$ is an approachable  time. 

By Theorem \ref{esssup}, we know there exists a sequence of approachable times $\{\tau_n: n=1,2,\cdots\}$
such that $\tau_*=\sup_{n} \tau_n$.  

Define $\tilde \tau_n=\max_{j=1}^n\{ \tau_j\}$, then $\tilde\tau_n\uparrow \tau_*$, every
 $\tilde\tau_n$ is 
 approachable, hence $\tau_*$ is approachable. 

\end{proof}

\subsection{Delimiting Time}\label{subsec-del-time} 
To construct a solution which also follows the spirit of second requirement, we start with the so called {\it naive time}
$\rho^{(1)}:=F(T).$

Roughly speaking, the naive time is the first time  $t$ when  immediate stopping is optimal according to the preference $J(\cdot;t)$. 
Stopping at the naive time does not consider possible changes in the future preferences (which may result in some earlier stopping), hence it is not in general a sophisticated solution.
However, if a  player reaches the naive time,  she should stop immediately even she has no sophisticated thinking, hence we call it as the $1^{st}$-order naive time (and this is what the upper label means). Furthermore, given the player has some sophisticated thinking, chopping off the time axis from $\rho^{(1)}$ will not change the problem at all. After the chopping, we can repeat the same operation by taking $\rho^{(1)}$ as the new time horizon. 
These  operations will lead to a backward iteration, which will hopefully approach a sophisticated solution. 


\begin{MyDef}[$n^{th}$-order Naive Time]
Denote $\rho^{(0)}=T$. For any $n\ge 1$, $\rho^{(n)}:=F(\rho^{(n-1)})$ is called the $n^{th}$-order naive time.
\end{MyDef}

Economically, we call a player $n^{th}$-order rational if she makes decision naively with the knowledge that  her solution will not go beyond $\rho^{(n-1)}$. 
For example, a $1^{st}$-order rational player makes decision naively, and a $2^{nd}$-order rational player knows that in the future she will play with $1^{st}$-order rationality, and she makes decision naively with this knowledge, hence  her solution will not go beyond $\rho^{(1)}$. With this terminology,  the $n^{th}$-order naive time  is the solution for a player with $n^{th}$-order rationality. 

Mathematically, $n^{th}$-order naive times define a decreasing sequence of stopping times. If we cut off the time set $\TS$ from the $n^{th}$-order naive time, the problem remains the same for  a player with $n^{th}$-order  rationality or with higher rationality, and this is why we also call all $n^{th}$-order naive times  delimiting. 
Furthermore, we define 
$$\rho^{(\infty)}:=\lim_{n\rightarrow +\infty}\rho^{(n)}.$$ 

Intuitively, $\rho^{(\infty)}$ is the solution for a player with infinite-order rationality, and hence can be a good candidate of our solution if it is a fixed point of $F$, i.e., an approachable time. Unfortunately, this is not true in general, and we will disprove it  by a counter example in Section \ref{sec6}. 

In case that $\rho^{(\infty)}$ is not a fixed point, we  need to continue the search for further delimiting time, which is a  solution for players with even higher order rationality.
In the rest of this subsection, we will find the limiting stopping time with delimiting property by the methods in \cite{lang2012algebra}. 

We firstly introduce an auxiliary concept called admissible set.

\begin{MyDef}[Admissible Set]\label{adm-set}
A subset of $\ST(0)$, denoted by $\bB$, is called an admissible set if:
\begin{itemize}
    \item [(i)] $T\in \bB$.
    \item [(ii)] Whenever $\rho \in \bB$, we have  $F(\rho) \in \bB$.
    \item [(iii)] For any totally ordered subset of $\bB$, its essential infimum is an element of $\bB$. 
\end{itemize}
\end{MyDef}

The following proposition states that intersection of two admissible sets is still an admissible set.

\begin{theorem}\label{close-to-int}
If  $\{\bB_\alpha\}_{\alpha\in A}$ is a family of admissible sets, then    $\cap_{\alpha\in A}\bB_\alpha$ is  an admissible set. 
\end{theorem}
\begin{proof}
Clearly, we know that $T\in \bB_\alpha$, so $T\in \cap_{\alpha\in A} \bB_\alpha$. This verifies [(i)] in Definition \ref{adm-set}.

For any $\rho \in \cap_{\alpha\in A}\bB_\alpha$, we know that $\rho \in \bB_\alpha$ for any  $\alpha\in A$. Since every $\bB_\alpha$ ism an admissible set, we have $F(\rho) \in \bB_\alpha$. Hence, $F(\rho)\in \cap_{\alpha\in A} \bB_\alpha$. This verifies [(ii)]. 

For any totally ordered subset of $\cap_{\alpha\in A}\bB_\alpha$, since $\bB_\alpha$ is an admissible set, we know its essential infimum is an element of  $\bB_\alpha$ . This verifies [(iii)].
\end{proof}

\begin{MyDef}[Delimiting Set and Delimiting Time]\label{delimit}
The intersection of all admissible sets, denoted by $\bM$,  is called the delimiting set. A stopping time $\rho$ is called a delimiting time if $\rho \in \bM$.
\end{MyDef}

By Theorem \ref{close-to-int},   we have the following interpretation of the delimiting set $\bM$.

\begin{prop}
The delimiting set $\bM$ is the smallest admissible set.
\end{prop}

Furthermore, we  show by the method in \cite{lang2012algebra} that the delimiting set is actually totally ordered, i.e, $\forall \tau_1, \tau_2\in \bM$, either $\tau_1\geq \tau_2$ or  $\tau_2 \geq \tau_1$. 


\begin{MyDef}[Extreme Point] \label{extreme-point}
A stopping time $\rho\in \bM$ is called an extreme point if 
    for $\tau \in \bM$, $\tau\geq  \rho$ and $\mathbb{P}(\tau >\rho)>0$ implies $F(\tau)\geq \rho$. 
\end{MyDef}

For an extreme point $\rho\in \bM$, define
\begin{equation}
    \bM(\rho)=\{\tau \in \bM: \mbox{either }\tau \geq \rho \mbox{ or } F(\rho) \geq \tau \}. 
\end{equation}

\begin{lemma}\label{extreme set}
    For every extreme point $\rho \in \bM$, we have $\bM(\rho)=\bM$. 
\end{lemma}
\begin{proof}
    Note that $\bM$ is the smallest admissible set, so we just need to prove that $\bM(\rho)$ is an admissible set.
    
    It is easy to see that $T \in \bM(\rho)$. So (i)  in Definition \ref{adm-set} is verified. 

    Let $\tau \in \bM(\rho)$. We verify that $F(\tau)\in \bM(\rho)$ in three cases. 
    \begin{itemize}
    \item [(1)] $\tau \geq \rho$ and $\mathbb{P}(\tau>\rho)>0$.  In this case,   by Definition \ref{extreme-point}, $F(\tau)\geq \rho$, 
    hence  $F(\tau) \in  \bM(\rho)$. 
    \item [(2)] $\tau=\rho$. In this case, $F(\tau)=F(\rho)$. Since $F(\rho) \in \bM(\rho)$, we have $F(\tau)\in \bM(\rho)$. 
    \item [(3)] $F(\rho)\geq \tau$. In this case,  $F(\rho)\geq \tau \geq F(\tau)$, hence 
     $F(\tau) \in \bM(\rho)$. 
     \end{itemize}

    Finally,  let $\bS$ be a totally ordered subset of $\bM(\rho)$ and let $\hat \tau$ be its essential infimum. Note that one can compare any element in $\bM(\rho)$ with $\rho$, 
    we just need to consider the following two cases:
    \begin{itemize}
    \item [(1)] $ \tau^{'}\geq \rho$ for all $\tau^{'} \in \bS$. In this case, $\hat \tau \geq \rho$ which shows that $\hat \tau \in \bM(\rho)$.
    \item [(2)] $\exists \tau^{''} \in \bS, F(\rho)\geq \tau^{''}$. In this case, $F(\rho)\geq \tau^{''} \geq \hat \tau$ which also shows $\hat \tau \in \bM(\rho)$. 
    \end{itemize}
\end{proof}

\begin{lemma}\label{extreme}
    Every element of $\bM$ is an extreme point. 
\end{lemma}
\begin{proof}
    Let $\hat \bM$ be the set of all extreme points of $\bM$. Similar to the proof of the previous lemma, it suffices to prove that $\hat\bM$ is an admissible set.

  $T\in \hat\bM$ is trivial. We turn to verify (ii) in Definition \ref{adm-set}. 
  Fix any extreme point $ \rho \in \hat\bM$, we need to show  that $F(\rho)$ is also an extreme point, i.e., 
  for any $\tau\in \bM$,  if $\tau\ge F(\rho)$ and $\mathbb{P}(\tau>F(\rho))>0$, then $F(\tau)\ge F(\rho)$. 
  
  Since $\tau\in \bM=\bM(\rho)$, we know 
there are only three cases for the comparison of $\tau$ and $\rho$:
\begin{itemize}
\item $\tau\ge \rho$ and $\mathbb{P}(\tau>\rho)>0$. In this case, $F(\tau)\ge \rho\ge F(\rho)$.
\item $\tau=\rho$. In this case, $F(\tau)=F(\rho)$. 
\item $\tau\le F(\rho)$. This is not possible given that $\tau\ge F(\rho)$ and $\mathbb{P}(\tau>F(\rho)$.
\end{itemize}
So (ii) is verified.

Finally, we verify (iii). 
    Let $\bS$ be a totally ordered subset of $\hat\bM$ and let $\hat \rho$ be its essential infimum of $\bS$. We  need to prove that $\hat \rho \in \hat\bM$. For any $\tau \in M$ with  $\tau \geq \hat \rho$ and $\mathbb{P}(\tau>\hat \rho)>0$, it will then suffice to prove that $F(\tau)\geq \hat \rho$.
    
    If for any $\tilde \rho\in \bS$,   $\tau\ge \tilde \rho$ does not hold, then $\tau\le F(\tilde\rho)\le \tilde\rho$, which implies 
    $\tau\le \hat\rho$. This contradicts the assumption ``$\tau \geq \hat \rho$ and $\mathbb{P}(\tau>\hat \rho)>0$''. 
    So there exists a $\tilde \rho\in \bS$ such that  $\tau\ge \tilde \rho$. 
    If in addition $\mathbb{P}(\tau>\tilde\rho)>0$, 
    then $F(\tau)\ge \tilde\rho\ge \hat\rho$. Otherwise,  $\mathbb{P}(\tau>\tilde\rho)=0$, in which case 
    $\tau=\tilde \rho$,  by the assumption ``$\tau \geq \hat \rho$ and $\mathbb{P}(\tau>\hat \rho)>0$'', we know 
    $\tilde\rho\ge \hat\rho, \mathbb{P}(\tilde\rho>\hat\rho)>0$, so $F(\tilde\rho)\ge \hat\rho$, which also implies that 
    $F(\tau)\ge \hat\rho$.     

\end{proof}

\begin{theorem}\label{totally ordered}
    $\bM$ is totally ordered. 
\end{theorem}
\begin{proof}
   For $\forall x,y \in \bM$, we know $x$ is an extreme point by Lemma \ref{extreme} and $y\in \bM(x)$ by Lemma \ref{extreme set}, so $y\geq x$ or $x\geq F(x)\geq y$. Therefore, $\bM$ is totally ordered.
\end{proof}

Define $$\rho^{*}:=\rm{essinf}\, \bM.$$  

\begin{prop}\label{be-delimit}
    $\rho^{*}$ is the smallest delimiting time.
\end{prop}
\begin{proof}
    Since $\bM$ is admissible and totally ordered, 
    we have $\rho^{*}\in \bM, \rho^{*}\leq \rho, \forall \rho \in \bM$. 
\end{proof}

By theorem \ref{totally ordered} and proposition \ref{be-delimit}, we know that $\bM$, although may not be countable,  can be sorted in a decreasing order  like $\{T, \rho^{(1)}, \rho^{(2)},...,\rho^{*}\}$. 
Hence we defined accurately the set of delimiting times and the last (minimal) delimiting time $\rho^*$. 

\subsection{Sophisticated Time}
Ideally, a sophisticated solution $\tau$ to our stopping problem should be 
both approachable and delimiting. 
In fact, it  turns out that the maximal approachable time equals the minimal delimiting time!

\begin{theorem}\label{max=min}
    $\rho^{*}=\tau_{*}$.
\end{theorem}
\begin{proof}
    To prove $\rho^{*}=\tau_{*}$, we firstly prove that $F(\rho^{*})=\rho^{*}$. Since $\bM$ is admissible, then there must be $F(\rho^{*})\leq \rho^{*}\in \bM$. Note that $\rho^{*}$ is the essential infimum in $\bM$, so $\rho^{*}\leq F(\rho^{*})$. Hence, $F(\rho^{*})=\rho^{*}$. This shows that $\rho^{*}\leq \tau_{*}$.

   To  prove that $\tau_{*}\leq \rho^{*}$, we consider the set $\Gamma:=\{\rho: \rho \in \bM, \rho \geq \tau_{*} \}$. We claim that $\Gamma=\bM$. To prove the claim, it suffices to show that $\Gamma$ is admissible. (i) in Definition \ref{adm-set} is trivially true. Note that $\forall \rho \in \Gamma$, $F(\rho)\geq F(\tau_*)=\tau_*$. Hence (ii)  is verified. Let $\bS$ be a totally ordered subset of $\Gamma$ and denote $\hat \rho$ as its essential infimum. Clearly, $\hat \rho \in \bM, \hat \rho \geq \tau^{*}$. Hence (iii) is verified.

   Therefore, $\rho^{*}\in \Gamma$ and so we have $\rho^{*}\geq \tau_{*}$.   
\end{proof}

By Theorem \ref{max=min}, we get the following relation between approachable times and delimiting times. 
\begin{coro}\label{acc-del-ineq}
For any approachable time $\tau$ and any delimiting time $\rho$, we have $\tau\le \rho$.
\end{coro}

Theorem \ref{max=min} also shows that $\tau_*=\rho^*$ is the only stopping time which is both approachable and delimiting, hence it is an ideal candidate of the sophisticated solution. 

\begin{MyDef}[Sophisticated Time] \label{equ-solution}
$\tau_*=\rho^*$ is called the sophisticated time. 
\end{MyDef}

\subsection{Backward Iteration}

The definition of sophisticated time provides two ways to find the equilibrium solution. One is the forward way to find the largest approachable time $\tau_*$, and the other is the backward way to find the smallest delimiting time $\rho^*$. 

We have no  general algorithm to find  the largest approachable time if there is no further mathematical structure of the problem. For the smallest delimiting time, since all $n^{th}$-order naive times $\rho^{(n)}$ are delimiting, so is $\rho^{(\infty)}$. 
When $\rho^{(\infty)}$ is approachable, we can find it by the backward iteration used in the definition of $n^{th}$-order naive times. 


\begin{theorem}\label{suff-soph}
If $\rho^{(\infty)}$ is approachable, then $\rho^{(\infty)}$ is the sophisticated time, i.e., $\rho^{(\infty)}=\tau_*=\rho^*$.
\end{theorem}

The condition in Theorem \ref{suff-soph} is quite hard to verify. While it holds trivially if there exists an integer $n$ such that 
$\rho^{(n+1)}=\rho^{(n)}$, in which case $\rho^{(\infty)}=\rho^{(n)}$. Furthermore, we have a nontrivial sufficient condition for $\rho^{(\infty)}$ to be approachable. 

\begin{theorem}\label{infinty-work}
Given that $J(\cdot; t)$ is lower semi-continuous in the right side for any $t$, in the sene that $\lim_{n\rightarrow \infty}J(\tau_n;t)\le J(\tau; t)$ for any decreasing sequence $\tau_n\downarrow \tau$. If for any integer $n$ and any $t< \rho^{(n )}$, there exists a $\tilde \rho \in\ST(t)$ with $\rho^{(n)}\le \tilde \rho\le \rho^{(n-1)}$, such that $J(\tilde \rho;t)\ge J(t;t)$, then $\rho^{(\infty)}$ is approachable. 
\end{theorem}
\begin{proof}
Fix any $t<\rho^{(\infty)}$. For any $n$, suppose $\tilde\rho_n$ is a stopping time with $\rho^{(n)}\ge \tilde\rho_n\ge \rho^{(n-1)}$ such that $J(\tilde \rho_n;t)\ge J(t;t)$,  then $\tilde\rho_n\downarrow \rho^{(\infty)}$, hence $J(\rho^{(\infty)};t)\ge \lim_{n\rightarrow \infty}J(\tilde\rho_n;t)\ge J(t;t)$, which implies that $\rho^{(\infty)}$ is approachable.
\end{proof}
\begin{re}
Notice that the statement in Theorem \ref{infinty-work} only works for the maximal version of the solution. If we choose the ``stop earlier'' principle, then the statement in this Theorem does not hold.
\end{re}

\section{Examples for the Backward Iteration} \label{sec6} 

In Section \ref{sec3}, we proposed the  $n^{th}$-order naive time $\rho^{(n)}$ by backward induction, and hoped that their limit 
$\rho^{(\infty)}$ is approachable and hence is a sophisticated solution.
Unfortunately, 
$\rho^{(\infty)}$ may or may not be approachable. 
In this section, we will show that both are possible by examples.  

\subsection{A Counter Example}
We start from a counter example, in which $\rho^{(\infty)}$ is not the sophisticated solution. 
\begin{Ex}\label{counter-jump}
Let ${\mathbb T}=[0,3]$. For any integer $k$, denote $t_k=1+\frac{1}{k}$. The flow of preferences $J$ is given as follows.
\begin{itemize}
\item [(i)] For any $t\ge 2$, $J(\tau;t)=2-\E[\tau]$.
\item [(ii)] For any integer $k\ge 1$, any $t\in [t_{k+1}, t_k)$, $J(\tau;t)=\E[|\tau-t_{k}|\id_{\tau\le t_k}]+\E[\tau\id_{\tau>t_k}]$.
\item [(iii)] For any $t\in [0, 1]$,  $J(\tau;t)=\E[(1-|\tau-\frac{1}{2}|)\id_{\tau\le 1}]+\E[\tau\id_{\tau>1}]$.
\end{itemize}
 \end{Ex}
 
Firstly, for any $t<1$, any stopping time $\tau>1$, it is easy to check that $J(t;t)=1-|t-1/2|<1/2$, while $J(\tau;t)=\E[\tau]>1$, hence $J(\tau;t)>J(t;t)$.

Let us calculate $\rho^{(n)}$ in this example step by step. 

\begin{itemize}
\item For $n=1$,  $t_n=t_1=2$. It is easy to check that $J(t_1;t_1)>J(\tau;t_1)$ for any stopping time $\tau>t_1$. So $\rho^{(1)}\le t_1$. 
For any $t\in [t_{n+1}, t_{n})$ for some integer $n\ge 1$,  $J(t;t)=t_n-t<t_n-t_{n+1}$, while $J(3;t)=3>J(t;t)$. Furthermore, 
  for any $t\le 1$, $J(t;t)<J(3;t)$.  Hence $\rho^{(1)}=t_1$.

\item For $n=2, t_2=\frac{3}{2}$. It is easy to check that $J(t_2;t_2)=t_1-t_2$, while  for any stopping time $\tau$ with 
$ t_2<\tau\le t_1$,
$J(\tau;t_2)=\E[|\tau-t_1|]=t_1-\E[\tau]<t_1-t_2$; so $\rho^{(2)}\le t_2$. 
For any $t\in [t_{m+1},t_m)$ with $m>n$, $J(t;t)=t_m-t<t_m-t_{m+1}$, while $J(t_1;t)=t_1>J(t;t)$. 
Furthermore,  for any $t<1$, $J(t;t)<J(t_1;t)$.  Hence $\rho^{(2)}=t_2$.

\item We can see that the previous analysis applies for $n=3,4,\cdots$, and by induction, we can prove that $\rho^{(n)}=t_n$. 
Hence $\rho^{(\infty)}=1$. 
\end{itemize}

It is easy to check that, at time $t=1/2$, $J(t;t)=1$; and for any stopping time $\tau$ with $t<\tau\le 1$, 
$J(\tau;t)=1+1/2-\E[\tau]<1$. So $J(t;t)>J(\tau;t)$ for any stopping time $\tau$ with $t<\tau\le 1$. 
Which means that $\rho^{(\infty)}=1$ is not approachable, and hence $\tau_*<\rho^{(\infty)}$!

\subsection{Non-exponential Discounting}\label{sec62}
In this subsection, we will stay in a system over the continuous time setting ${\mathbb T}=[0, +\infty)$. 
The $d$-dimensional state process $X_\cdot$ of the system is given by the stochastic differential equation 
\begin{equation}\label{system}
dX_{s}=b(s,X_{s})ds+\sigma(s,X_{s})dB_{s},\qquad X_{0}=x_0,
\end{equation}
where $x_0$ is a fixed initial state  at time $t=0$,
 $b(\cdot,\cdot)$ and $\sigma(\cdot,\cdot)$ are given deterministic functions 
$B_\cdot$ is a standard Brownian motion. The filtration of the system is  
the augmented natural filtration of $B_\cdot$, i.e., $\cF_t=\cF^B_t\bigvee \mathcal{N}(\PP)$, which satisfies the usual condition, i.e., it is right continuous and universally complete. 

For any $t\in {\mathbb T}$, the preference at time $t$ is to maximise $$J(\tau; t)=\E_t[D(\tau-t)g(X_\tau)]$$ 
over $\tau\in \ST(t),$
where $\E_t=\E[\cdot|\cF_t]$, 
 $D: \R^+\mapsto (0,1]$ 
is the discounting function, which is  decreasing and $D(0)=1$, 
and $g: \R^d\mapsto \R^+$ is the reward function. 

In this subsection, we keep the following assumptions on the parameters of the problem.
\begin{assump}\label{con-hd}
\begin{itemize}
	\item [(i)] The SDE (\ref{system}) admits a unique strong solution, which is strongly Markovian. 
	\item [(ii)] $g(\cdot)$ is nonnegative  and continuous. 	
	\item [(iii)] $D(t)D(s)\le D(t+s)$ for any $s\ge 0$ and $ t\ge 0$. 
	\item [(iv)] $Y_t:=D(t)g(X_t)$ is a process of class ${\mathbf D}$, i.e., $\{Y_\tau\}_{\tau\in \ST(0)}$ is uniformly integrable.
\end{itemize}
\end{assump}

\begin{re}
In this assumption, (i) and (ii) are very mild and often assumed in classical optimal stopping problems;
(iii) is a critical condition, which is usually called decreasing impatience in literature on non-exponential discounting; 
and (iv) is used to ensure $J(\tau;t)$ to be  right continuous in $\tau$, see the next Lemma.
\end{re}

\begin{lemma}\label{cont-J}
Given Assumption (\ref{con-hd}).(iv), $J(\tau;t)$ is continuous in $\tau$, i.e., If $\{\tau_n\}_{n=1,2,\cdots}$ is a convergent sequence of stopping times, $\tau=\lim_{n\rightarrow +\infty}\tau_n$,   then $J(\tau;t)=\lim_{n\rightarrow +\infty}J(\tau_n;t)$.
\end{lemma}
\begin{proof}
Fix $t$. For any $s\ge t$, denote $Z_s:=D(s-t)g(X_s)$. Since $D(s-t)\le \frac{D(s)}{D(t)}$, hence 
$Z_s\le \frac{1}{D(t)}[D(s)g(X_s)]=\frac{Y_s}{D(t)}$. With $Y$ being in class ${\mathbb D}$, so is $Z$, 
hence for any convergent sequence $\{\tau_n\}$ with  $\tau=\lim_{n\rightarrow +\infty}\tau_n$,  
\begin{eqnarray*}
J(\tau;t)&=&\E_t[Z_{\tau}]\\
&=&\E_t[\lim_{n\rightarrow +\infty}Z_{\tau_n}]\\
&=&\lim_{n\rightarrow +\infty}\E_t[Z_{\tau_n}]\\
&=&\lim_{n\rightarrow +\infty}J(\tau_n;t).
\end{eqnarray*}
\end{proof}

As proposed in Section \ref{sec3}, we find the sophisticated solution by $n^{th}$-order naive times. Denote $\rho^{(0)}=+\infty$, then for any $n\ge 1$, the $n^{th}$-order naive time $\rho^{(n)}$ is calculated inductively by 
$$\rho^{(n)}=\inf\{t\ge 0: t\le \rho^{(n-1)}, \mbox{ and } g(X_t)>\E_t[D(\tau-t)g(X_\tau)] \mbox{ for }\forall \,\tau\in \ST(t)\ \mbox{ with } t<\tau\le \rho^{(n-1)} \}.$$
By  Lemma \ref{cont-J-0} and its proof, $\rho^{(n)}$ is a stopping time for any integer $n$.  Denote $\rho^{(\infty)}=\lim_{n\rightarrow+\infty}\rho^{(n)}$. 

To study the sequence of high order strategies, we need some properties of pre-committed optimal stopping problem involved 
in the calculation of $\rho^{(n)}$. 
For any stopping time $\tau_0$, any stopping time $\bar\tau>\tau_0$, consider the optimal stopping problem 
\begin{equation}\label{opt-tau}
\max_{\tau\in \ST(\tau_0), \tau\le \bar\tau} \E_{\tau_0}[D(\tau-\tau_0)g(X_\tau)].
\end{equation}
 
\begin{lemma}\label{exist-opt-tau}
With Assumption (\ref{con-hd}).(iv), Problem (\ref{opt-tau}) admits optimal stopping times. 
\end{lemma}
\begin{proof}
Pennanen and Ari-Pekka \cite{pennanen2018} studied the existence of optimal solution for the optimal stopping problem 
$\max_{\tau\in\ST(0)} \E[R_\tau]$ with a given stochastic reward process $R$. They call the process $R$ regular
 if it is of class $\mathbf{D}$ and its left-continuous version $R_{-}$ is indistinguishable with its predictable projection $^p\hspace{-0.1cm} R$.
 Given the filtration $\cF_\cdot$ satisfying the usual condition, they proved that the optimal stopping problem admits optimal solutions (Theorem 4 in \cite{pennanen2018}).
 
For our Problem (\ref{opt-tau}), by the strong Markovian property of the state process $X_\cdot$, we know the problem is equivalent to 
\begin{equation}\label{exist-opt-equiv}
\max_{\tau\in \hat\ST(0), \tau\le \tilde\tau} \E[D(\tau)g(X_\tau)|X_0=x]
\end{equation}
  with 
$x=X_{\tau_0}$ and $\tilde \tau$ representing the shift of $\bar\tau$ due to the initial condition  $X_0=x$. 

For Problem (\ref{exist-opt-equiv}),  the filtration is assumed to satisfy the usual condition. $R_s=D(s)g(X_s)$ is of class $\mathbf{D}$, nonnegative,  and continuous in time. 
Furthermore, if we define $\hat R_t=R_{t}\id_{t\le \tilde\tau}-\id_{t>\tilde \tau}$, then $\hat R_t$ is left-continuous, implying that $\hat R$ is regular. Hence 
the problem $\max_{\tau\in\ST(\tau_0)} \E[\hat R_\tau]$ admits optimal solutions. Denote by $\tau^*$ an optimal solution, then 
$\tau^*\le \tilde \tau$, otherwise, over the event $\tau^*>\tilde\tau$, $\hat R_{\tau^*}=-1<\hat R_{\tilde\tau}$, and 
$\tau^*\wedge \bar\tau$ is strictly better than $\tau^*$, which contradicts the optimality of $\tau^*$. So $\tau^*$ is also an optimal stopping time for Problem (\ref{exist-opt-equiv}).
\end{proof}

With this lemma, the following corollary is obvious.
\begin{coro}\label{pre-com-exists}
For any integer $n$ and any stopping time $\tau_0$ with $\tau_0\le \rho^{(n)}$,  the optimal stopping problem 
\begin{equation}\label{pre-opt-n}
\max_{\tau\in \ST(\tau_0), \tau\le \rho^{(n)}} \E_t[D(\tau-\tau_0)g(X_\tau)]
\end{equation}
 admits  pre-committed optimal solutions.
\end{coro}

Furthermore, we have the following properties of the set of optimal solutions of Problem (\ref{pre-opt-n}).
\begin{prop}\label{max-opt-n}
If $\tau_1$ and $\tau_2$ are both optimal stopping times for Problem (\ref{pre-opt-n}), so are $\bar \tau:=\tau_1\vee\tau_2$ and $\underline\tau:=\tau_1\wedge \tau_2$. 

If $\{\tau_n\}_{n=1,2,\cdots}$ is a monotone sequence of stopping times,  then $\lim_{n\rightarrow +\infty}\tau_n$ is also an optimal stopping time for the same problem. 
\end{prop}
\begin{proof}
Denote $v$ as the optimal value of Problem (\ref{pre-opt-n}). Define 
$$\begin{array}{cc}
a_1=\E_{\tau_0}[D(\tau_1-\tau_0)g(X_{\tau_1})\id_{\tau_1\ge \tau_2}], & b_1=\E_{\tau_0}[D(\tau_1-t)g(X_{\tau_1})\id_{\tau_1< \tau_2}];\\
a_2=\E_{\tau_0}[D(\tau_2-t)g(X_{\tau_2})\id_{\tau_1\ge \tau_2}], & b_2=\E_{\tau_0}[D(\tau_2-t)g(X_{\tau_2})\id_{\tau_1< \tau_2}].
\end{array}
$$
Its is easy to see that 
$$a_1+b_1=\E_{\tau_0}[D(\tau_1-\tau_0)g(X_{\tau_1})]=v=\E_{\tau_0}[D(\tau_2-\tau_0)g(X_{\tau_2})]=a_2+b_2$$
and $$a_1+b_2=\E_{\tau_0}[D(\bar\tau-\tau_0)g(X_{\bar\tau})], \quad a_2+b_1=\E_{\tau_0}[D(\underline \tau-\tau_0)g(X_{\underline \tau})].$$
Hence $\E_{\tau_0}[D(\bar\tau-\tau_0)g(X_{\bar\tau})]+\E_{\tau_0}[D(\underline \tau-\tau_0)g(X_{\underline \tau})]=2v.$
Together with the fact that 
$$\E_{\tau_0}[D(\bar\tau-\tau_0)g(X_{\bar\tau})]\le v, \mbox{  and } \E_{\tau_0}[D(\underline\tau-\tau_0)g(X_{\underline\tau})]\le v,$$
 we know that both inequalities above are equalities. 
 
 The second statement is obvious with Assumption \ref{con-hd}.(iv). 
\end{proof}

With Corollary \ref{pre-com-exists} and Proposition \ref{max-opt-n}, we know that there exists a maximal optimal stopping time $\bar s_n$, which is the essential supremum of all optimal stopping times. 

Now we claim the main conclusion in this subsection. 

\begin{prop}\label{exists-s-n}
For any $n\ge 1$, any $t<\rho^{(n+1)}$, there exists a stopping time $\tau$ with $\rho^{(n+1)}\le \tau\le \rho^{(n)}$, 
such that $J(\tau;t)\ge J(t;t)$. 
\end{prop}

\begin{proof}
Since $t<\rho^{(n+1)}\le \rho^{(n)}$,  by the definition of $\rho^{(n+1)}$, there exists a $\tau\in \ST(t)$ with $\tau\le \rho^{(n)}$ and $\tau>t$, such that 
$\E_t[D(\tau-t)g(X_\tau)]\ge \E_t[D(t-t)g(X_t)]=g(X_t)$. 

By Corollary \ref{pre-com-exists} and Proposition \ref{max-opt-n}, we have the maximal optimal stopping time $\bar s_n$ for Problem (\ref{pre-opt-n}).
Hence $\bar s_n>t$ and 
$$\E_t[D(\bar s_{n}-t)g(X_{\bar s_n})]\ge g(X_t).$$

Now we claim that $\bar s_n\ge \rho^{(n+1)}$.  If this is not true, we denote $\Gamma:=\{\omega: \bar s_n <\rho^{(n+1)}\}$, then $\PP_t(\Gamma)>0$, and $\Gamma\in \cF_{\bar s_n \wedge \rho^{(n+1)}}$.
 
Over $\Gamma$, since $\bar s_n<\rho^{(n+1)}\le \rho^{(n)}$, there exists a stopping time $\hat\tau>\bar s_n$ with $\hat\tau\le \rho^{(n)}$, such that 
 $g(X_{\bar s_n})\le \E_{\bar s_n} [D(\hat\tau-\bar s_n)g(X_{\hat\tau})]$, hence 
\begin{eqnarray*}
&&\E_t[D(\bar s_n-t)g(X_{\bar s_n})]\\
&=&\E_t[D(\bar s_n-t)g(X_{\bar s_n})\id_\Gamma]+\E_t[D(\bar s_n-t)g(X_{\bar s_n})\id_{\Gamma^c}]\\
&\le&\E_t[D(\bar s_n-t) \E_{\bar s_n} [D(\hat\tau-\bar s_n)g(X_{\hat\tau})] \id_\Gamma]+\E_t[D(\bar s_n-t)g(X_{\bar s_n})\id_{\Gamma^c}]\\
&=&\E_t[D(\bar s_n-t) D(\hat\tau-\bar s_n) g(X_{\hat\tau}) \id_\Gamma]+\E_t[D(\bar s_n-t)g(X_{\bar s_n})\id_{\Gamma^c}]\\
&\le &\E_t[D(\hat\tau-t) g(X_{\hat\tau}) \id_\Gamma]+\E_t[D(\bar s_n-t)g(X_{\bar s_n})\id_{\Gamma^c}]\\
&=&\E_t[D(\bar\tau-t) g(X_{\bar\tau})
\end{eqnarray*}
with $\bar\tau=\bar s_n\id_{\Gamma^c}+\hat\tau\id_{\Gamma}\le \rho^{(n)}$. Hence $\bar \tau$ must be an optimal stopping time for 
Problem (\ref{pre-opt-n}), and $\bar\tau\ge \bar s_n, \PP(\bar\tau>\bar s_n)>0$, which contradicts the fact that $\bar s_n$ is the maximal optimal stopping time for Problem (\ref{pre-opt-n})! So we conclude that   $\bar s_n\ge \rho^{(n+1)}$ is true, with which we have 
$$\rho^{(n+1)}\le \bar s_n \le \rho^{(n)}\mbox{ and }\E_t[D(\bar s_n-t)g(X_{\bar s_n})]\ge g(X_t).$$
\end{proof}

\begin{theorem}
$\rho^{(\infty)}$ is the sophisticated solution.
\end{theorem}

\begin{proof}
With Proposition 
\ref{exists-s-n}, and Theorem \ref{infinty-work}, we only need to prove that $\E_t\left[D(\tau-t)g(X_{\tau})\right]$ is lower semicontinuous from the right hand side. By Lemma \ref{cont-J}, $J(\cdot;t)$ is continuous in $\tau$ 
when we have  Assumption \ref{con-hd}.(iv). 

\end{proof}

Furthermore, we have the following strong property of $\rho^{(\infty)}$  from Proposition  \ref{exists-s-n}. 
\begin{prop}\label{opt-rho-inf}
For any time $t<\rho^{(\infty)}$, $J(\rho^{(\infty)};t)=\max_{\tau\in \ST(t), \tau\le \rho^{(\infty)}}J(\tau;t)$. 
\end{prop}
\begin{proof}
Given a time $t<\rho^{(\infty)}$, for any integer $n$, by Proposition \ref{exists-s-n}, there exists a stopping time 
$s_n$ with $\rho^{(n+1)}\le s_n\le \rho^{(n)}$, and $s_n$ is optimal for $\max_{\tau\in \ST(t), \tau\le \rho^{(n)}} J(\tau;t)$.
Hence, for any stopping time $\tau\in\ST(t)$ with $\tau\le \rho^{(\infty)}$, $J(\tau;t)\le J(s_n;t)$. 

Since $s_n\rightarrow \rho^{(\infty)}$, by Lemma \ref{cont-J}, we know $J(\rho^{(\infty)};t)=\lim_{n\rightarrow +\infty} J(s_n;t)\ge J(\tau;t)$, which proves the statement in this proposition.
\end{proof}

For the non-exponential discounting problem in this subsection,  we only claim that $\rho^{(\infty)}$ is the sophisticated time. 
It can be very hard to calculate $\rho^{(n)}$ explicitly, especially  in multi-dimensional cases. 
In the next subsection, we will show by an example that explicit calculation is possible in  $1$-dimensional case with $b(\cdot,\cdot)$ and $\sigma(\cdot, \cdot)$ being  time independent.

\subsection{A Detailed Example of Non-Exponential Discounting}
In this section, we study the optimal stopping with the objective at time $t$
\begin{equation}\label{ex-non-expo}
J(\tau;t)=\E_t\left[\frac{|B_\tau|}{1+\beta (\tau-t)}\right]
\end{equation}
with the discounting function $D(t)=\frac{1}{1+t}$ 
where $\beta$ is a positive constant, and  $B_\cdot$ is a $1$-dimensional standard Brownian motion. 

Fitting Problem (\ref{ex-non-expo}) into the discussion in Section \ref{sec62}, we have 
$D(t)=\frac{1}{1+\beta t}, \mu(t,x)=0, \sigma(t,x)=1, g(x)=|x|$. 
So Assumption (\ref{con-hd}).(i,ii, iii) hold. 
We need to check the uniform integrability of $\frac{|B_t|}{1+\beta t}$.

\begin{lemma}\label{Y-class-D}
$Y_t:=\frac{|B_t|}{1+\beta t}$ is of {{class ${\mathbf D}$.}}
\end{lemma}
\begin{proof}
Denote $Z_t=Y_t^2$, then 
$$dZ_t=\left(\frac{1+\beta t-2\beta B_t^2}{(1+\beta t)^3}\right)dt+\frac{2B_t}{(1+\beta t)^2}dB_t.$$
For any stopping time $\tau$ and finite time $T>0$, 
\begin{eqnarray*}
\E[Z_{\tau\wedge T}]&=&\E\left[\int_0^{\tau\wedge T}\left(\frac{1+\beta s-2\beta B_s^2}{(1+\beta s)^3}\right)ds\right]\\
&<&\E\left[\int_0^{\tau\wedge T}\left(\frac{1}{(1+\beta s)^2}\right)ds\right]\\
&\le&\E\left[\int_0^{+\infty}\left(\frac{1}{(1+\beta s)^2}\right)ds\right]\\
&:=&M<+\infty.
\end{eqnarray*}
Hence $\E[Z_\tau]=\E[\lim_{T\rightarrow +\infty} Z_{\tau\wedge T}]\le \lim_{T\rightarrow +\infty} E[Z_{\tau\wedge T}]\le M $
which implies that $Y$ is of class $\mathbf{D}$. 
\end{proof}

With this lemma, we can apply all conclusions in Section \ref{sec62} for Problem (\ref{ex-non-expo}), construct high order policies $\{\rho^{(n)}\}_{n=1,2,\cdots}$ and its limit $\rho^{(\infty)}$, and we know $\rho^{(\infty)}$ is the sophisticated solution. 

We calculate $\rho^{(\infty)}$ by Proposition 
 \ref{opt-rho-inf}, which claims that at any time $t<\rho^{(\infty)}$, $\rho^{(\infty)}$ is optimal for the static {{problem 
$\max_{\tau\in\ST(t), \tau\le \rho^{(\infty)}} \E_t\left[\frac{|B_\tau|}{1+\beta (\tau-t)}\right]$.}}

With this property, our definition of maximal approachable time corresponds to an equilibrium policy defined in Huang and Nguyen-Huu \cite{huang2018time}, and it achieves the highest $J(\tau; 0)$ over all approachable times. 
In the rest of this subsection, we  borrow the following lemma \ref{Huang-NH}  from \cite{huang2018time} to make an explicit calculation of $\rho^{(\infty)}$. 

It is easy to see that Problem (\ref{ex-non-expo}) is not only Markovian, but also time homogeneous, hence 
the $1^{st}$-order naive time   $\rho^{(1)}=\inf\{t\ge 0: |B_t|\in A_1\}$ for some closed set $A_1\subset \R^+$. 
Similarly, $$\rho^{(n)}=\inf\{t\ge 0: |B_t|\in A_n\}$$
for some closed set $A_n\subset \R^+$ with $A_n\supset A_{n-1}$, 
  $\rho^{(\infty)}=   \inf\{t\ge 0: |B_t|\in {\rm cl}(A_{\infty}) \}$ with $A_\infty=\cup_n A_n$ and ${\rm cl}(A_\infty)$ is the closure of $A_\infty$\footnote{In fact, by  $\rho^{(\infty)}:=\lim_{n\rightarrow +\infty}\rho^{(n)}$, we should conclude that $\rho^{(\infty)}=   \inf\{t\ge 0: |B_t|\in A_{\infty} \}$. Furthermore, by Lemma 4.4 in \cite{huang2020optimal}, we can replace $A_\infty$ with $\rm{cl}(A_\infty)$. 
  }. 

Since $|B_0|=0$ and $\rm{cl}(A_\infty)\subset \mathbb{R}^+$, we know $\rho^{(\infty)}=\inf\{t>0: |B_t|=b\}=:\tau_b$ with $b=\inf A_\infty$. 
In fact, by Proposition \ref{opt-rho-inf}, and the fact that $\rho^{(\infty)}$ is the largest approachable time, we know that $b$ is also the biggest number  such that, 
for any $t<\tau_b:=\inf\{t>0: |B_t|=b\}$, $\tau_b$ is optimal for $$\max_{\tau\in \ST(t), \tau\le \tau_b} \frac{|B_\tau|}{1+\beta(\tau-t)}.$$
Equivalent, with the notation $B^x_t:=B_t+x$ and $\tau^x_b:=\inf\{t>0: |B^x_t|=b\}$, then $b$ is the biggest number such that 
for any $x$ with $|x|<b$, $\tau^x_b$  is optimal for 
$$\max_{\tau\in\ST(0), \tau\le \tau^x_b } J(\tau; x,b):=\E\left[\frac{|B^x_\tau|}{1+\beta \tau}\right].$$

For any real number $a>0$ and $x$ with $0<x\le a$, define  $\eta(x,a):=\E\left[\frac{a}{1+\beta \tau^x_a}\right]=J(\tau^x_a; x,a)$. We have the following lemma from Huang and Nguyen-Huu \cite{huang2018time}.

\begin{lemma}\label{Huang-NH}
\begin{itemize}
	\item [(1)] $\eta(x,a)$ is strictly increasing and strictly convex in $x$ over $x\in [0,a]$, and $0<\eta(0,a)<\eta(a,a)=a$.
	\item [(2)] There exists a unique $a^*\in (0, \beta^{-1/2})$ such that
	\begin{itemize}
		\item [(2.a)] $\forall\, a>a^*$, equation $\eta(x,a)=x$ admits  a unique solution $x^*(a)$ over the interval $x\in (0, a^*)$, hence $\eta(x,a)>x$ for any $x<x^*(a)$ and $x>\eta(x,a)$ for any $x>x^*(a)$. 
		\item [(2.b)] $\forall\, a\le a^*$, $\eta(x,a)>x$ for any $x\in (0, a)$.  
	\end{itemize}
\end{itemize}
\end{lemma}

With $a^*$ in this lemma, we claim that $b$ in $\rho^{(\infty)}$ is exactly $a^*$. 
\begin{theorem}
$\rho^{(\infty)}=\tau_{a^*}$.
\end{theorem}
\begin{proof}
For any $t<\tau_{a^*}$, denote $x=B_t$, we have $|x|<a^*$, hence $J(t;t)=|x|<\eta(x,a^*)=J(\tau_{a^*};t)$, which means that $\tau_{a^*}$ is approachable. 

For any $a>a^*$, by (2.a) in Lemma \ref{Huang-NH}, we can uniquely find $x^*(a)\in (0, a^*)$ to solve $\eta(x,a)=x$. 
and for any $x\in (x^*(a), a)$, we have $\eta(x,a)<a$. For any $t<\tau_a$ with $x:=|B_t|>x^*(a)$, we have $x\in (x^*(a), a)$, 
and hence $\eta(x,a)<x$. 
Standing at time $t$, we have $J(t; t)=x, \eta(x,a)=J(\tau_a; t)$,  so $\rho^{(\infty)}=\tau_a$ is impossible given that $\rho^{(\infty)}$ satisfies Proposition  \ref{opt-rho-inf}. 

With the fact that $\rho^{(\infty)}$ must be in the form of $\tau_b$, we proved that $b=a^*$ and hence 
$$\rho^{(\infty)}=\tau_{a^*}.$$
\end{proof}

\section{Comparison between BIS and  Sophisticated Stopping Strategy}\label{sec5}
In a finite discrete time setting, the BIS is the mostly accepted equilibrium solution. Does our sophisticated stopping strategy respect
the BIS in a finite discrete time setting?  In this section, we will remain in a finite discrete time setting with 
${\mathbb T}=\{0,1,\cdots, T\}$ as in Section \ref{sec2}, and compare the two solutions.

In the definition of  BIS, at any time $t$, we  only compare the value of immediate stopping $J(t;t)$ and  that of the next  `optimal' stopping time $J(\tau^b_{t+1}; t)$; while in the definition of our sophisticated stopping strategy, we need to compare  $J(t;t)$ against $J(\rho;t)$ for any stopping time $\rho$ with $t<\rho\le \rho^{(n)}$. This difference tends to cause the BIS to stop earlier than the 
sophisticated strategy, and this is true as stated in the following theorem. 

 \begin{theorem}\label{bisvssophist}
The sophisticated strategy $\tau_{*}$ is no earlier than the BIS $\tau^b_0$, i.e., $\tau_*\ge \tau^b_0$. 
 \end{theorem}
 \begin{proof}
 By Theorem \ref{bisprop2}, we know $\forall s<\tau^b_{0}$,
 \begin{equation}
     J(\tau^b_{0};s )\ge J(s;s). 
 \end{equation}
This implies that  $\tau^b_{0}$ is an approachable time. With the fact that $\tau_*$ is the biggest approachable time, 
we have  $\tau^b_{0}\le \tau_{*}$.
 \end{proof}

The inequality in Theorem \ref{bisvssophist} does not fully support the possibility that  $\tau^b_0<\tau_*$ is possible. We show this possibility by the following example.  

\begin{Ex}\label{bis-smaller}
(An example for $\tau^{b}_0< \tau^{*}$)
Take $T=5$. The preference $J(\cdot;t)$ is specified by 
\begin{itemize}
\item [(i)] $J(\tau; 5)=J(\tau;4)=\E[5-\tau]$;
\item [(ii)] $J(\tau;t)=\E[\tau]$ for $t=1,2,3$;
\item [(iii)] $J(\tau;0)=5-\E[|\tau-1|]$. 
\end{itemize}
It is easy to check that $\tau^b_5=5; \tau^b_4=4; \tau^b_3=\tau^b_2=\tau^b_1=4; \tau^b_0=0$;
while $\rho^{(1)}=4, \rho^{(2)}=\rho^{(1)}=4$, hence $\tau_*=\rho^{(\infty)}=\rho^{(1)}=4$. 
 \end{Ex}
 
 \begin{re}
 Example \ref{bis-smaller} also shows that $\tau^b_0$ can be not the maximal approachable time. 
 \end{re}

With the difference shown in Theorem \ref{bisvssophist} and Example \ref{bis-smaller}, a natural question is that 
which solution sounds more reasonable?  We cannot simply answer this question without a criterion for the comparison,
but in the rest of this section, we show that  the sophisticated solution has the following good property, while the BIS does not. 

\begin{MyDef} \label{monotone}
Given a flow of preference $J(\cdot; s)$ over $s\in \{0,1,2,\cdots, T\}$, denote $\tau_{0, t}$ as the solution, in some sense, for the problem over the time interval $s\in \{0,1,2,\cdots, t\}$ under some definition of a solution. 
We say the solution $\{\tau_{0,t}: t\in \{1,2,\cdots,T\}\} $ is forward monotonic if $\tau_{0,t}\le \tau_{0, t+1}$ for any $t\in \{1,\cdots,T-1\}$.
\end{MyDef}

\begin{re}
   The Definition \ref{monotone} of forward monotonicity can be easily extended to any time setting.  
\end{re}

\begin{re}
The forward monotonicity sounds reasonable, which says that, 
if the time horizon of the stopping problem is shortened from the end, then the solution should stop earlier. 
When preferences are time consistent, the (maximal) optimal solution is clearly forward monotonic. While it is much more complicated when  the preferences are not time consistent.
\end{re}

For the sophisticated solution, we have the monotonicity in Definition \ref{monotone}.
\begin{theorem}
The sophisticated time in Definition \ref{equ-solution}  is monotone. 
\end{theorem}
\begin{proof}
Denote by $\tau_*(t)$ the sophisticated solution over the time setting $s\in \{0,1,\cdots, t\}$ for any $t\le T$. 
Then $\tau_*(t)$  is the maximal approachable time over $s\in \{0,1,\cdots, t\}$, and is also approachable over the time set 
$s\in \{0,1,\cdots, t+1\}$, hence $\tau_*(t)\le \tau_*(t+1)$. 
\end{proof}
\begin{re}
    The proof above can be extended naturally for general time setting. 
\end{re}

While for the BIS, we don't have the monotonicity, which can be seen by the following  counter example. 

\begin{Ex}
Take $T=5$, and the preference $J$ are given by 
\begin{itemize}
\item [(i)] $J(\tau;5)=J(\tau;4)=J(\tau;3)=\E[\tau]$;
\item [(ii)] $J(\tau;2)=\E[|\tau-3.2|]$;
\item [(iii)] $J(\tau;1)=5-\E[|\tau-2|]$;
\item [(iv)] $J(\tau;0)=\E[\tau]$.
\end{itemize}
Denote $\tau^b(t)$ as the BIS for the problem over the time set $\{0,1,\cdots, t\}$. 
It is easy to check that $\tau^b(5)=1$, while $\tau^b(4)=2$, which denies the monotonicity of the BIS. 

As a comparison, denote $\tau_*(t)$ as the sophisticated solution over the time set $\{0,1,\cdots, t\}$. 
Then $\tau_*(5)=5$ and $\tau_*(4)=2$, which supports the monotonicity of the sophisticated solution. 
\end{Ex}

\section{Concluding Remark}\label{sec8}

Dynamic optimal stopping problem without time consistency is hard, not only because of the loss of dynamic principle, but also because of the lack of widely accepted definition of a solution for general settings. In this paper, we propose a sophisticated solution for a general optimal stopping problem without time consistency, where the time setting can be continuous,  discrete, or hybrid, and the preferences can be in any form. Our definition respects the game-theoretical idea for dealing with dynamic optimisation without time consistency, and it is defined both forwardly and backwardly, which meet up at our final solution. The forward definition is 
hard to analyse without detailed structure of the problem, while the backward definition can be approached by a sequence of stopping times under some sufficient condition. The sufficient condition  does not hold in general, which is disproved by a counter example. While for the widely studied non-exponential discounting problem, we prove that the sufficient condition does hold. 

For the uniqueness of our definition of sophisticated solution, we find that, given the principle that ``stop later if  the best  future stopping shares the same objective value as immediate stopping'', our solution is unique, which corresponds to the maximal optimal stopping time in the time consistent case. There is also a (unique) minimal version of the sophisticated solution obtained with the principle that 
``stop immediately  if it achieves the maximal objective value''.

\bibliography{refs}{}
\bibliographystyle{spmpsci}

\end{document}